\documentclass[12pt, reqno]{amsart}

\allowdisplaybreaks

\usepackage{soul}

\usepackage{amsmath}

\usepackage{amssymb}

\usepackage{amsfonts}

\usepackage{verbatim,xspace,titletoc}

\usepackage[usenames]{color}

\usepackage[margin=1.2in]{geometry}

\usepackage[pdftex]{graphicx}

\usepackage[utf8]{inputenc}

\usepackage[usenames,dvipsnames]{xcolor}

\definecolor{cobalt}{rgb}{0.0, 0.28, 0.67}

\definecolor{darkblue}{rgb}{0.0, 0.0, 0.55}

\usepackage[pagebackref,colorlinks,linkcolor=BrickRed,citecolor=darkblue,urlcolor=black,hypertexnames=true]{hyperref}

\makeatletter

\def\moverlay{\mathpalette\mov@rlay}

\def\mov@rlay#1#2{\leavevmode\vtop{

                \baselineskip\z@skip \lineskiplimit-\maxdimen

                \ialign{\hfil$#1##$\hfil\cr#2\crcr}}}

\makeatother

 \newtheorem{theorem}{Theorem}[section]

 \newtheorem{corollary}[theorem]{Corollary}

 \newtheorem{lemma}[theorem]{Lemma}

\theoremstyle{definition}

\newtheorem{definition}[theorem]{Definition}

\newtheorem{example}[theorem]{Example}

\theoremstyle{remark}

\newtheorem{fact*}{Fact}

\DeclareMathOperator{\GL}{GL}

\DeclareMathOperator{\SL}{SL}

\DeclareMathOperator{\QM}{QM}
\DeclareMathOperator{\M}{M}

\DeclareMathOperator{\diag}{diag}

\newcommand{\N}{\mathbb{N}}

\newcommand{\C}{\mathbb{C}}

\newcommand{\FF}{\mathbb{F}}

\newcommand{\inv}{^{-1}}

\newcommand{\til}{\raise.17ex\hbox{$\scriptstyle\mathtt{\sim}$}}

\newcommand\cD{\mathcal D}

\newcommand\cH{\mathcal H}

\newcommand\beq{\begin{equation}}

\newcommand\eeq{\end{equation}}

\newcommand{\bbm}{\left[ \begin{smallmatrix}}

\newcommand{\ebm}{\end{smallmatrix} \right]}

\newcommand{\bpm}{\left( \begin{smallmatrix}}

\newcommand{\epm}{\end{smallmatrix} \right)}

\numberwithin{equation}{section}

\newcommand{\dfn}[1]{{\bf #1}\index{#1}}

\newcommand{\plangle}{\moverlay{(\cr<}}

\newcommand{\prangle}{\moverlay{)\cr>}}

\newcommand{\fpolys}[2]{\mathbb{F}{<}{#1}_1,\ldots, {#1}_{#2}{>}}

\newcommand{\fpolysinf}[1]{\mathbb{F}{<}{#1}_1,{#1}_2,\ldots{>}}

\newcommand{\unram}{unramified\xspace}

\newcommand{\Unram}{Unramified\xspace}

\newcommand{\fpolysm}[1]{\mathbb{F}\!<\!{#1}\!>\!}

\newcommand{\frats}[2]{\mathbb{F}\plangle{#1}_1,\ldots, {#1}_{#2}\prangle}

\newcommand{\fratsm}[1]{\mathbb{F}\plangle #1\prangle}

\title[Invariant noncommutative rational functions]{Noncommutative rational functions invariant under the action of a finite solvable group}

\author[I. Klep]{Igor Klep${}^1$}
\address{Igor Klep, Department of Mathematics, Faculty of Mathematics and Physics, University of Ljubljana, Slovenia}
\email{igor.klep@fmf.uni-lj.si}
\thanks{${}^1$Supported by the 
	Slovenian Research Agency grants J1-8132, N1-0057 and P1-0222. 
	Partially supported by the 
	Marsden Fund Council of the Royal Society of New Zealand.}

\author[J. E. Pascoe]{James Eldred Pascoe${}^2$}
\address{James E. Pascoe, Department of Mathematics, University of Florida, USA}
\email{pascoej@ufl.edu}
\thanks{${}^2$Partially supported by NSF MSPRF DMS 1606260.}

\author[G. Podlogar]{Gregor Podlogar}
\address{Gregor Podlogar, 
	Institute of Mathematics, Physics and Mechanics,  Ljubljana,
	Slovenia}
\email{gregor.podlogar@imfm.si}

\author[J. Vol\v{c}i\v{c}]{Jurij Vol\v{c}i\v{c}${}^3$}
\address{Jurij Vol\v{c}i\v{c}, Department of Mathematics, Texas A\&M University, USA}
\email{volcic@math.tamu.edu}
\thanks{${}^3$Supported by the NSF grant DMS 1954709.}

\date{\today}



\begin{document}
	
\def\fratc#1#2{\mathbb{C}\plangle{#1}_1,\ldots, {#1}_{#2}\prangle}
\def\fratcm#1{\mathbb{C}\plangle #1\prangle}
\def\fpolyc#1#2{\mathbb{C}{<}{#1}_1,\ldots, {#1}_{#2}{>}}
\def\ort{\operatorname{U}_d(\mathbb{C})}

\begin{abstract}
This paper describes the structure of invariant skew fields for 
linear actions of finite solvable groups on free skew fields in $d$ generators. These invariant skew fields are always finitely generated, which contrasts with the free algebra case.
For abelian groups or solvable groups $G$ with a well-behaved representation theory it is shown that the invariant skew fields are free on $|G|(d-1)+1$ generators. 
Finally, positivity certificates for invariant rational functions in terms of sums of squares of invariants are presented.
\end{abstract}

\maketitle

\tableofcontents

\section{Introduction}

Classical invariant theory studies polynomials that are preserved under linear group actions \cite{Kra84,Stu08,DK15}. 
By the Chevalley--Shepard--Todd theorem \cite[Theorem 2.4.1]{Stu08}
for a finite group $G\subseteq\GL_n(\C)$, the ring of invariants 
$\C[x_1,\ldots,x_n]^G$
is isomorphic to a polynomial ring (in the same number of variables) 
if and only if $G$ is a complex reflection group. Similarly, 
one considers the rational invariants $\C(x_1,\ldots,x_n)^G$. {Noether's problem} asks when this invariant field is rational, that is,
isomorphic to a field of rational functions. This is a subtle question which heavily depends on the structure of the group; however, in recent years much progress has been made along the lines of
\cite{Sal84,C-TS07,Pey08,CHKK10,Mor12,CHHK15,JS+}.
Examples of invariant fields give the simplest negative answers to the {L\"uroth problem}, i.e., examples of unirational varieties which are not rational \cite{AM72,Sal84}. L\"uroth's problem has a positive answer in one variable (every field between $K$ and $K(x)$ must be $K$ or 
purely transcendental over $K$), and in two variables over $\C$. In complex analysis, these problems pertain to complex automorphisms and holomorphic equivalence of domains, geometry of symmetric domains and realizations of symmetric analytic functions \cite{GR08,satake2014algebraic,AY17}. On the practical side, symmetries are regularly applied in control system design to analyze a system by decomposing it into lower-dimensional subsystems \cite{GM85,vdS87,Kwa95}.

We study the free noncommutative analogue of the above program over an algebraically closed field $\mathbb{F}$ of characteristic $0.$
Let $x = (x_1 \ldots, x_d)$ be a tuple of noncommuting indeterminates.
A \dfn{noncommutative polynomial} is a formal linear combination of words in $x$ with coefficients in  $\mathbb{F}.$
For example,
	$$17x_1^4 + 13x_1x_2 - 9x_2x_1 + 39.$$
We denote the free associative algebra of noncommutative polynomials on $d$ generators by $\fpolys{x}{d}.$
A \dfn{noncommutative rational expression} is a syntactically valid combination of noncommutative polynomials, arithmetic operations $+, \cdot, {}^{-1},$ and parentheses, e.g.
$$\big(216x_1^3x_2^4x_1^5-((x_1x_2 - x_2x_1)^{-1} + 3)^8\big)^{-1}.$$
These expressions can be naturally evaluated on $d$-tuples of matrices. An expression is called nondegenerate if it is valid to evaluate it on at least one such tuple of matrices. Two nondegenerate expressions with the same evaluations wherever they are both defined are equivalent.
A \dfn{noncommutative rational function} is an equivalence class of a nondegenerate rational expression. They form the \dfn{free skew field} $\frats{x}{d}$, which is the universal skew field of fractions of the free algebra $\fpolys{x}{d}$. We refer the reader to \cite{BGM05,Coh06,HMV06,BR11,KVV12,Vol18}
for more on the free skew field.

We analyze the invariants in a free skew field
under the action of a finite solvable group.
For example, a symmetric noncommutative rational function $r$ in two variables satisfies the equation
$$r(x,y) = r(y,x).$$
Naturally, this corresponds to an action of the symmetric group with two elements, $S_2$, and we denote
the ring of symmetric noncommutative rational functions in two variables by
$\fratsm{x,y}^{S_2}.$ Similarly, $\fpolysm{x,y}^{S_2}$ is
the ring of symmetric noncommutative polynomials. By a theorem of Wolf, see e.g. \cite[\S 6.8]{Coh06}, we have
$$\fpolysm{x,y}^{S_2} \cong \fpolysinf{u}.$$
In fact, the polynomial invariants $\fpolys{x}{d}^G$ are, except in some trivial cases, always isomorphic to a free algebra in countably infinitely many variables \cite[\S 6.8]{Coh06}. As was observed by Agler and Young \cite{AY14} (see also \cite{CPT-D18} and \cite{AMY18}),
$$\fpolysm{x,y}^{S_2} 
\subseteq
\fratsm{x +y, (x-y)^2, (x-y)(x+y)(x-y)},$$
where
$\fratsm{x+y, (x-y)^2, (x-y)(x+y)(x-y)}$ denotes the subfield of 
$\fratsm{x, y}$ generated by $x+y, (x-y)^2,$ and $(x-y)(x+y)(x-y).$
What is perhaps less clear, and follows from our
Theorem \ref{abelianisomorphism}, is that
\begin{align*}
\fratsm{x,y}^{S_2} & =\fratsm{x+y, (x-y)^2, (x-y)(x+y)(x-y)} \\
& \cong \fratsm{a, b, c}.
\end{align*}
Moreover, the isomorphism $\varphi: \fratsm{a, b, c} \rightarrow \fratsm{x,y}^{S_2}$
satisfies $\varphi(a) = x+y,$ $\varphi(b) = (x-y)^2,$
and $\varphi(c) = (x-y)(x+y)(x-y).$

The equality and further isomorphism are remarkable for a few reasons.
First, in the noncommutative case, it is nontrivial to show that the set of symmetric noncommutative polynomials generate the free skew field of symmetric rational functions. For example, expressions for
relatively simple symmetric rational functions may require complicated expressions in terms of the generators, as is shown by the equalities expressing
$x^{-1} + y^{-1}:$
\begin{align*}
		x^{-1} + y^{-1} & = 4(x+y - (x-y)^2((x-y)(x+y)(x-y))^{-1}(x-y)^2)^{-1} \\
		& = \varphi((a - bc^{-1}b)^{-1})
\end{align*}
(It is somewhat hard to even find an elementary way of showing the equality;  we manufactured it using realization theory which will be a key ingredient of the proof of our main result.)
Secondly, it is interesting that $x+y, (x-y)^2, (x-y)(x+y)(x-y)$ satisfy no hidden rational relations, which follows from a result of Lewin \cite[Theorem 1]{Lew74}. Unlike in the commutative case, it does not suffice to test only polynomial relations; for example, $x,xy,xy^2$ satisfy no polynomial relations and generate a proper free subalgebra in $\fpolysm{x,y}$, while they satisfy a rational relation and the skew field they generate in $\fratsm{x,y}$ is $\fratsm{x,y}$ itself. The theory of symmetric noncommutative functions was first initiated through quasideterminants in \cite{GKLLRT95}, and their combinatorial aspects were further studied in \cite{RS06,BRRZ08}. For the construction of a noncommutative manifold corresponding to symmetric analytic noncommutative functions, and associated Waring--Lagrange theorems and Newton--Girard formulae, see \cite{AMY18}.

There are also potential applications of understanding invariant subfields of the free skew field in theoretical control theory. Noncommutative rational functions naturally arise as transfer functions of linear systems that evolve along free semigroups \cite{BGM05,BGM06}. When such a system admits additional symmetries (described by a group action), so does its associated transfer function. If the generators for this group action are known, then the transfer function can be expressed in terms of invariant building blocks. Thus it is likely that, as with classical discrete-time linear systems, this decomposition of the transfer function leads to a decomposition of the linear system into lower-dimensional subsystems, which make for a simpler analysis.

\subsection{Main results}

Let $G \subseteq \GL_d(\FF)$ be a finite group.
The \dfn{skew field of rational invariants}, denoted by $\frats{x}d^G$, is the skew field of elements of $\frats xd$
that are invariant under the action of $G,$ that is, 
\[
\frats{x}d^G=\{r\in\frats xd \colon 
r(g\cdot x) = r(x) \text{ for all }g \in G\}.
\]

Our first main result states that 
for solvable groups  the 
skew field of rational invariants
 is always finitely generated:

\begin{theorem}\label{t:solv}

Let $G\subseteq \GL_d(\FF)$ be a finite solvable group. Then 
the skew field of invariants
 $\frats{x}d^G$ is finitely generated.

\end{theorem}

For solvable groups $G$ with a well-behaved
representation theory we can give a finer structure of the invariant skew field $\frats xd^G$.

\begin{definition}

Let $G$ be a finite group.

\begin{enumerate}

\item Let $N$ be a nontrivial normal abelian subgroup. We say that $G$ is \dfn{\unram} over $N$ if
	for every irreducible representation $\pi$ of $G,$ $\pi|_N$ is trivial or $\pi|_N$ splits into distinct irreducible representations
	of $N.$

\item We say a group $G$ is \dfn{totally \unram} if either $G$ is the trivial group or there exists a nontrivial normal abelian subgroup $N$
such that $G$ is \unram over $N$ and $G/N$ is totally \unram.

\end{enumerate}

\end{definition}

\begin{theorem}\label{thm:unram}

Let $G$ be a totally \unram group acting on $\mathbb{F}^G$ via the left regular representation.
Then
$$\frats{x}{|G]}^G \cong \frats{u}{|G|(|G|-1)+1}.$$

\end{theorem}

Examples of totally \unram groups include abelian groups, $S_3,$ $S_4,$ and dihedral groups; furthermore, all groups of order $<24$ are totally \unram.
For these, noncommutative Noether's problem is tractable in the sense that for left regular representations the answer is affirmative. In fact, we conjecture that 
invariant skew fields of finite
groups are always free.

\def\Z{\mathbb Z}
The smallest non-example of a totally \unram group is the group $\SL_2(\FF_3)$ of order $24$. The next non-examples are given by eight groups of order $48$:
the four groups in the isoclinism class of $(\Z_4\times D_4)\rtimes\Z_2$,
the isoclinism class of $\GL_2(\mathbb F_3)$ with two groups,
and the isoclinism class of $\Z_2\times\SL_2(\mathbb F_3)$ again containing two groups. Then we find 
32 groups that  are not totally \unram
among the 267 groups of order 64. The smallest odd order examples are found among groups with 243 elements.

Finally, we present applications of the above results (for $\FF=\C$) to real algebraic geometry. We provide positivity certificates for invariant noncommutative rational functions in terms of invariant weighted sums of squares. We say that $r\in \fratc{x}{d}$ is {\bf positive} if for every $n\in\N$ and a tuple of hermitian matrices $X\in\M_n(\C)^d$, $r$ is defined at $X$ and $r(X)$ is a positive semidefinite matrix. The following is a solution of the invariant free rational Hilbert's 17th problem.

\begin{theorem}\label{c:h17}
Let $G \subset\ort$ be a finite solvable group.
Then there are $q_1, \ldots, q_N \in \fratc{x}{d}^G$ such that
every positive rational function $r \in \fratc{x}{d}^G$ is of the form
\[r = \sum_j \tilde{s}_j^* q_{n_j} \tilde{s}_j\]
for some $\tilde{s}_j \in \fratc{x}{d}^G$.
\end{theorem}

Furthermore, we establish Positivstellens\"atze for invariant semialgebraic sets of bounded operators on a separable Hilbert space. Corollary \ref{prop:posSS} treats strict positivity when the invariant constraints satisfy an Archimedean condition, and Corollary \ref{prop:ConvexPosSS} certifies positivity on convex domains (i.e., those given by linear matrix inequalities).

\subsection{Reader's guide}

The paper is organized as follows. After Section \ref{sec:prelim} with preliminaries we establish
Theorem \ref{t:solv} in Section \ref{sec:solvable}.
Theorem \ref{thm:unram} for abelian groups $G$
is proved
in Section \ref{sec:abelian}, followed
by the proof of the theorem itself,
and a strengthening thereof (Theorem \ref{thm:unram2}) in 
Section \ref{sec:unram}. 
Finally, Section \ref{sec:pos} discusses
relationships with real algebraic geometry;  positivity certificates for invariant positive rational functions
can be chosen to be invariant themselves.

\subsection*{Acknowledgment}
The authors thank Primo\v z Moravec and John Shareshian for insightful discussions on finite groups, and the anonymous referee for suggestions that vastly improved the presentation of the paper.

\section{Preliminaries on group representations}\label{sec:prelim}

In this section we give some background and introduce notions which will be necessary for the sequel.

\subsection{Pontryagin duality}

Let $N$ be a finite abelian group.
Define $N^*$ to be the group of multiplicative homomorphisms $\chi: N \rightarrow \mathbb{F
}^*.$
The group $N^*$ is non-canonically isomorphic to $N.$
Every representation $\pi$ of $N$ decomposes into a direct sum of elements of $N^*,$ that is, $N^*$ consists of all the irreducible representations
of $N.$ For more details see \cite{Ser77,Rud90}.

\subsection{Complete representations}\label{ssec:complete_representations}

A faithful representation $\pi$ of a group $G$ is {\bf complete} if 
there is a direct summand $\pi_B$ of $\pi$ (i.e., $\pi$ decomposes as $\pi_B \oplus \pi_J$ for some subrepresentation $\pi_J$)	
and there is a nontrivial normal abelian subgroup $N\subseteq G$ such that:

\begin{enumerate}

	\item

$\pi_B|_N$ contains exactly the nontrivial representations of $N$ as direct summands with multiplicity $1$;

\item The representation 
\[\pi_{N^\tau}
\oplus (\pi_{B}\otimes \pi \oplus \pi \otimes \pi_B)_{N^\tau}
\oplus (\pi_{B}\otimes \pi \otimes \pi_B)_{N^\tau}\]
is a complete representation of $G/N$.
Here, for a representation $\varrho$,
$\varrho_{N^\tau}$ denotes the summands of $\varrho$
which are trivial on $N$ and thus naturally gives rise to a representation
of $G/N$.

\end{enumerate}

The notion of a complete representation is rather technical; the proper motivation is unveiled in Lemma \ref{lem:linear}, where completeness ensures linearity of certain induced group actions. In any case, complete representations should be viewed as a companion concept to the more natural definition of an \unram group.
Namely, as seen in the proof of Theorem \ref{thm:unram} below, the left regular representation of a totally \unram group is complete.

\subsection{Unramified groups}\label{ssec:unram}

The interplay between subgroups and representations is the subject of Clifford theory, see e.g.~\cite{Isa76}. We now give a reinterpretation of what it means for $G$ to be \unram over a normal abelian subgroup $N.$
There is a natural action of $G/N$ on $N^*$ given by
\[
gN: \chi \mapsto (n\mapsto \chi(g^{-1}n g))=\chi^g.
\]
Let $\pi$ an irreducible representation of $G,$ such that $\pi|_N$ decomposes as $\bigoplus_i \chi_i.$ For any $gN \in \textrm{Stab } \chi_i,$
we have that $\pi(g)e_i \in \textrm{span } e_i,$
where $e_i$ form the basis corresponding to the decomposition
of $\pi|_N$ into one-dimensional representations
$\bigoplus_i \chi_i.$ That is, $G/N$ acts on the characters composing $\pi$, so if they are all distinct, as in the case of a totally \unram group, it permutes them.

\begin{example}

We show that $S_4$ is totally \unram.
The maximal abelian normal subgroup is given by $\{e, (12)(34), (13)(24), (14)(23)\}$ which is an isomorphic copy of the Klein four group
$V = \mathbb{Z}_2 \times \mathbb{Z}_2.$
The character tables of the representations of $V$ and $S_4$ are given by:
\begin{center}
\begin{tabular}{cc}
\begin{tabular}{c|rrrrr}
&&&&\\ \hline
$\pi_1$&1&1&1&1\\
$\pi_2$&1&-1&1&-1\\
$\pi_3$&1&1&-1&-1\\
$\pi_4$&1&-1&-1&1\\
\end{tabular}
\qquad & \qquad
\begin{tabular}{c|ccccc}
&\{\}&\{2\}&\{2,2\}&\{3\}&\{4\}\\ \hline
$\tau_1$&1&1&1&1&1\\
$\tau_2$&1&-1&1&1&-1\\

$\tau_3$&3&1&-1&0&-1\\
$\tau_4$&3&-1&-1&0&1\\
$\tau_5$&2&0&2&-1&0\\
\end{tabular} \\[1mm]
\phantom{V}$V$ & \phantom{VV}$S_4$\\[2mm]
\end{tabular}
\end{center}
\iffalse
The character table of the representations of $V$ is given by:
\begin{center}
\begin{tabular}{c|rrrrr}
&&&&\\ \hline
$\pi_1$&1&1&1&1\\
$\pi_2$&1&-1&1&-1\\
$\pi_3$&1&1&-1&-1\\
$\pi_4$&1&-1&-1&1\\
\end{tabular}
\end{center}
The character table of $S_4$ is:
\begin{center}
\begin{tabular}{c|ccccc}
&\{\}&\{2\}&\{2,2\}&\{3\}&\{4\}\\ \hline
$\tau_1$&1&1&1&1&1\\
$\tau_2$&1&-1&1&1&-1\\

$\tau_3$&3&1&-1&0&-1\\
$\tau_4$&3&-1&-1&0&1\\
$\tau_5$&2&0&2&-1&0\\
\end{tabular}
\end{center}
\fi
The only irreducible characters of $S_4$ that are nontrivial on $V$ are $\tau_3$ and $\tau_4$. 
From the tables we see $\tau_3|_V=\tau_4|_V=\pi_2+\pi_3+\pi_4$.
Therefore $S_4$ is \unram  over $V$.

Now it remains to see that $S_4/V \cong S_3$ is totally \unram. 
The group $S_3$ has three irreducible representations, two of which are one-dimensional. 
The  two-dimensional representation restricted to the $\mathbb{Z}_3$ subgroup has two distinct characters on the diagonal.
\end{example}

\begin{example}
The dihedral groups $D_{n}=\langle a,b \colon a^n =  b^2 =abab= e\rangle$ are totally \unram. The irreducible two-dimensional representations $\pi$ are given by 
$$
\pi \colon a \mapsto \begin{pmatrix}
\omega & 0\\
0 & \omega\inv
\end{pmatrix},
\ b \mapsto
\begin{pmatrix}
0 & 1\\
1 & 0
\end{pmatrix}
$$ 
where $\omega$ is a primitive $n$-th  root of unity. The restriction of $\pi$ to $\langle a \rangle$ splits into two distinct irreducible representations.
For $n \geq 4$ such a representation is clearly not complete.

\end{example}

\section{Solvable groups and their invariants}\label{sec:solvable}

In this section we prove Theorem \ref{t:solv}.
A main technical ingredient are realizations of 
noncommutative rational functions, i.e., a canonical-type forms for them.

\subsection{Realizations}
%A convenient way of presenting a noncommutative rational function is via realization theory.
Each rational function can be written in the form
\beq \label{realizationformula} r = c^*L^{-1}b \eeq
where $b,c\in \mathbb{F}^n,$ and 
\[L = A_0 + \sum^d_{i=1} A_ix_i\]
for some $A_i \in \M_n(\mathbb{F}).$ This formula is nondegenerate if and only if $L$ admits an invertible matrix evaluation.
For a comprehensive study of noncommutative rational functions we
refer to \cite{Coh06,BR11} or \cite{BGM05,KVV12}.
We will need the realization formula in \eqref{realizationformula} to prove for abelian groups 
(and thus for solvable groups via a later inductive argument) that the 
noncommutative polynomial invariants generate the rational invariants.

\subsection{Proof of Theorem \ref{t:solv}}

We prove that if $G/H$ is abelian and $\frats{x}{d}^{H}$ is finitely generated,
then $\frats{x}{d}^{G}$ is finitely generated.
Note that this suffices for  proving  Theorem \ref{t:solv} 
by a simple
inductive argument.

\begin{lemma}\label{lem:fg}

Let $H$ be a normal subgroup of $G$ such that $G/H = N$ is abelian.
Suppose there are finitely many $q_i \in \frats{x}{d}^{H}$ such that every $r \in \frats{x}{d}^H$ is of the form
\begin{equation}\label{eq:relReal}
r(x) = c^*\left(A_0 + \sum_i A_i q_i(x)\right)^{-1}b,
\end{equation}
where the formula on the right-hand side is nondegenerate.
Then, there exist finitely many $\tilde{q}_{j} \in \frats{x}{d}^G$ 
such that for every $\tilde{r}\in \frats{x}{d}^G$ we have 
\begin{equation}\label{eq:relRealNew}
\tilde{r}(x) = \tilde{c}^*\left(\tilde{A}_0 + \sum_j \tilde{A}_j \tilde{q}_j(x)\right)^{-1}\tilde{b},
\end{equation}
and the formula on the right-hand side is nondegenerate.
\end{lemma}

\begin{proof}
Let $V = \textrm{span } \{n \cdot q_i(x) \colon n \in N\}.$ By Pontryagin duality there exists a basis $\{v_i\}_i$ for $V$ such that $n\cdot v_i = \chi_i (n)v_i.$
Without loss of generality, $q_i = v_i.$
For each nontrivial $\chi$ in the subgroup of $N^*$ generated by the $\chi_i,$ pick a
monomial $m_\chi$ in the $q_i$ such that $n\cdot m_\chi = \chi(n) m_{\chi}.$  Without loss of generality, the subgroup generated by the $\chi_i$ is the whole of $N^*.$

The representation $ \chi_i \mapsto \oplus_{n \in N} \chi_i(n)$ is conjugate to the left regular representation of $N^*$. 
Denote $\tilde{\chi}_i = P(\oplus_{n \in N} \chi_i(n))P\inv$ where $\tilde{\chi}_i$ is the permutation matrix that maps $e_\nu$ to $e_{\chi_i\nu}$. Define vectors 
$$
s^*=\bpm 1 & 0 &\cdots &0 \epm P \quad \text{and} \quad t= P\inv
%\bpm 1 \\ 0 \\\vdots \\0 \epm.
\bpm 1 & 0 &\cdots &0 \epm^*.
$$
Index the rows of vectors $s, t$ with elements of $N$. Observe that $s^*t=\sum_{n \in N}s_nt_n=1$.\looseness=-1

Let $\tilde{r}(x)$ be a $G$-invariant rational function. Then it is in particular $H$-invariant, and by assumption it admits a realization as in \eqref{eq:relReal}.
From it one derives a new realization of $\tilde{r}(x)$,
\begin{align*}
\tilde{r}(x) &= 
\sum_{n \in N}s_nt_nc^*\left(A_0 + \sum_i A_i q_i(n \cdot x)\right)^{-1}b
\\ &=
(s \otimes c)^*
\left(I \otimes A_0 + \oplus_{n\in N}\sum_i A_i q_i(n\cdot x)\right)\inv
(t \otimes b)\\	 &
= (s \otimes c)^*
\left(I \otimes A_0 + \oplus_{n\in N}\sum_i A_i \chi_ {i}(n)q_i(x)\right)\inv
(t \otimes b)\\
&= \bpm c \\ 0
		\\
	\vdots
		\\ 
	0\epm^*
	\left(I \otimes A_0 + \sum_i \tilde{\chi}_i \otimes A_i q_i(x)\right)^{-1}
	\bpm b \\ 0
		\\
	\vdots
		\\ 
	0\epm \\
&=: \tilde{c}^*
	\left(I \otimes A_0 + \sum_i \tilde{\chi}_i \otimes A_i q_i(x)\right)^{-1}
	\tilde{b}.
\end{align*}
Let $M$ be a diagonal matrix with  columns indexed by elements of $N^*$ (starting with the identity)
such that the diagonal entries are $m_{\chi}$ for nontrivial $\chi$ and $1$ otherwise. Similarly, let
$\hat{M}$ be diagonal with diagonal entries $m_{\chi^{-1}}$ for nontrivial $\chi$ and $1$ otherwise. 
Observe that $\tilde{c}^*(M \otimes I)=\tilde{c}^*$ and $(\hat{M} \otimes I)\tilde{b}=\tilde{b}$

Then,
\[\begin{split}
	r(x)&= \tilde{c}^*
	\left(I \otimes A_0+ \sum_i \tilde{\chi}_i \otimes A_i q_i(x)\right)^{-1}
	\tilde{b} \\
	&= \tilde{c}^*(MM^{-1} \otimes I)
	\left(A_0^{\oplus |N|} + \sum_i \tilde{\chi}_i \otimes A_i q_i(x)\right)^{-1}(\hat{M}^{-1}\hat{M} \otimes I)
	\tilde{b} \\
	&= \tilde{c}^*
	\left(\hat{M}M \otimes A_0 + \sum_i \hat{M}\tilde{\chi}_iM  q_i(x)  \otimes A_i\right)^{-1}
	\tilde{b} .
\end{split}\]
Note that $\hat{M}{M}$ is invariant. 
Now we need to show that $\hat{M}\tilde{\chi}_iq_i(x) M$ is invariant. 
The nonzero elements of $\hat{M}\tilde{\chi}_iq_i(x) M$ are
$m_{(\chi_i\nu)^{-1}} q_i(x) m_{\nu}$
which are clearly invariant by the choice of $q_i$ and $m_\chi$.
Let $\tilde{q_j}\in \frats{x}{d}^G$ be the nonconstant entries of matrices $\hat{M}{M}$ and $\hat{M}\tilde{\chi}_iq_i(x) M$ (they do not depend on $b,c,A_i$ from \eqref{eq:relReal}). So there are constant matrices $\tilde{A}_j$ such that
$$\hat{M}M \otimes A_0 + \sum_i \hat{M}\tilde{\chi}_iM  q_i(x)  \otimes A_i =
\tilde{A}_0 + \sum_j \tilde{A}_j \tilde{q}_j(x).$$
This concludes the proof since the new form \eqref{eq:relRealNew} is defined wherever $M$ was invertible, and thus non-degenerate.
\end{proof}

\section{The abelian case}\label{sec:abelian}

The next theorem
shows that the invariant fields for abelian groups
are free and can be explicitly described.

\begin{theorem} \label{abelianisomorphism}

Let $G\subset \GL_d(\FF)$ be abelian. Then
$$
\frats{x}{d}^G \cong \frats{u}{|G|(d-1)+1}.
$$
If $G$ is diagonal, that is $G = \oplus \chi_i$ where $\chi_i \in G^*$, 
then any minimal set of generators
for the subgroup of the free group on $d$ generators given by
the words 
$x_{i_1}^{j_1}\ldots x^{j_k}_{i_k}$ with $\chi_{i_1}^{j_1}\ldots \chi^{j_k}_{i_k} = 1$
can serve as the preimage of the $u_j$.

\end{theorem}

\begin{proof}
We note that every linear action of an abelian group $G$ can be diagonalized with an appropriate linear change of coordinates. Hence there exist linear polynomials $w_1, \ldots, w_d$ such that
$$g \cdot w_i = \chi_i(g)w_i,$$
where $\chi_i$ belongs to the character group of $G,$ denoted $\hat{G}$, and $w_i$ form an orthonormal basis (in the sense that the coefficients are orthogonal) for the space of all linear polynomials in
$\fpolys{x}{d}.$
By \cite[Theorem 7.4]{CPT-D18}
the elements of $\fpolys{x}{d}^G$ are spanned
by monomials of the form $w_{i_1}\ldots w_{i_k}$
such that $\chi_{i_1}\ldots \chi_{i_k} = 1$.
By embedding $\fpolys{x}{d}$ into the group algebra
of the free group on $d$ generators by mapping the $w_i$ to the said group generators, one obtains that 
the invariants in the free group algebra (of noncommutative Laurent polynomials)
are generated by $|G|(d-1)+1$ elements via the Nielsen-Schreier theorem \cite[Section I.3]{LS01} as in \cite[Theorem 7.5]{CPT-D18}.
Concretely, we have a surjective homomorphism from the free
group with $d$ generators to $\hat{G}$
which is itself non-canonically isomorphic to $G.$
The kernel of this homomorphism is a subgroup of the free group with $d$ generators, which must be free and have $|G|(d-1)+1$  generators via the Nielsen-Schreier formula.
The generating elements can satisfy no rational relations by
\cite{Lew74} (see also \cite{Lin00}), that is, their rational closure is a free skew field on these  generators.

Therefore it suffices to show that 
polynomial invariants
$\fpolys{x}{d}^G$ generate
the skew field of rational invariants $\frats{x}{d}^G$.
This follows a similar line of reasoning as used in the proof of Lemma \ref{lem:fg}.
Let $r \in \frats{x}{d}^G.$
By the realization theory,
one can write any element of the free skew field as
$r = c^*(A_0 + \sum_i A_i w_i)^{-1}b.$
Now $g\cdot r = r$, so as in the proof of Lemma \ref{lem:fg}, 
we have 
\[ r = c^*\left(A_0 + \sum_i A_i \chi_i(g)w_i\right)^{-1}b
 = \bpm c \\ 0
\\
\vdots
\\ 
0\epm
^*\left(I \otimes A_0 + \sum_i \tilde{\chi}_i \otimes A_i w_i\right)^{-1}
\bpm b \\ 0
\\
\vdots
\\ 
0\epm,
\]
 where $\tilde{\chi}_i$ is the permutation matrix that maps $e_\nu$ to $e_{\chi_i\nu}$.

Fix polynomials $v_\chi$ such that $g \cdot  v_\chi = 
\chi(g)v_\chi$ and $v_\tau = 1$, where $\tau$ is the trivial representation. 
Let $V$ be a diagonal matrix whose diagonal entries are the $v_{\chi}$.
Similarly, let $\hat{V}$ be diagonal with diagonal entries $v_{\chi^{-1}}$.
Now
\begin{align*}
 r &= \bpm c \\ 0
\\
\vdots
\\ 
0\epm^*\left(I\otimes A_0 + \sum_i \tilde{\chi}_i\otimes A_i w_i\right)^{-1}
\bpm b \\ 0
\\
\vdots
\\ 
0\epm \\
 &= \bpm c \\ 0
\\
\vdots
\\ 
0\epm^*(VV^{-1}\otimes I)\left(I\otimes A_0 + \sum_i \tilde{\chi}_i\otimes A_i w_i\right)^{-1}(\hat{V}^{-1}\hat{V}\otimes I)
\bpm b \\ 0
\\
\vdots
\\ 
0\epm \\
&= \bpm c \\ 0
\\
\vdots
\\ 
0\epm^*\left(\hat{V}V\otimes A_0+ \sum_i (\hat {V}\tilde{\chi}_iw_i V)\otimes A_i\right)^{-1}
\bpm b \\ 0
\\
\vdots
\\ 
0\epm.
\end{align*}

As in the proof of Lemma \ref{lem:fg},
 $\hat{V}V$ and  $\hat{V} \tilde{\chi}_iw_i V$
are invariant under the action of $G$.
 We get
$(\tilde{\chi}_i)_{\eta, \nu} = 1$ if $\eta = \chi_i \nu$
and $0$ otherwise.
Now, $(\hat{V} \tilde{\chi}_iw_i V)_{\chi_i\nu, \nu} = v_{\chi_i^{-1} \nu^{-1}}w_iv_{\nu}$ is clearly invariant, so we are done.
\end{proof}

\begin{corollary}\label{cor: abelbasis}
	Let $G$ be an abelian group acting on $\mathbb{F}^d$ via a complete representation $\pi=\pi_B\oplus \pi_J,$
	where $\pi_B$ acts on $\mathbb{F}^{G^*\setminus \{\tau\}}$ and $\tau$ denotes the trivial representation. 
	Let $b_\chi$ and $j_i$ diagonalize $\pi_B$ and $\pi_J$, respectively.
		Then
		$$
	\frats{x}{d}^G \cong \frats{u}{|G|(d-1)+1},
		$$
		and
	 preimages of the $u_j$ are of the form 
		$b_\chi b_\eta b_{(\chi \eta)^{-1}}, b_\chi j_i b_{(\chi\eta_i)\inv}$, where $b_\tau=1$.

\end{corollary}

\begin{proof}
For $v_\chi$ from the proof of Theorem \ref{abelianisomorphism} we take $b_\chi$, while for $w_i$ we take $b_\chi$ and $j_i$ for $i=1,\dots, d-|G|+1$, where $d-|G|+1$ is the dimension of $\pi_J$. Clearly
$$
\left\{
b_\eta b_{\eta\inv},\ b_\chi b_\eta b_{(\chi\eta)\inv},\ b_\chi j_i b_{(\chi\eta_i)\inv}
\colon \chi\in G^*,\ \eta \in G^*\setminus\{\tau\},\ 1\le i\le t
\right\}\setminus\{1\}$$
generate $\frats{x}{d}^G$. Since $b_\eta b_{\eta\inv}=b_\tau b_\eta b_{\eta\inv}=b_\eta b_{\eta\inv} b_\tau$, there are
$$(|G|-1)+(|G|-1)(|G|-2)+ |G|(d-|G|+1)=|G|(d-1)+1$$
generators. By \cite[Corollary 5.8.14]{Coh95} they are free generators of the free skew field of invariants.
\end{proof}

\begin{example}
Let $\omega$ be a third root of unity and $c$ a generator of $\mathbb{Z}_3$. 
Define a representation of $\mathbb Z_3$ on $\mathbb{F}^2$ by $cx=\omega x$ and $cy=\omega^2y$. Then we have 
$$
\fratsm{x,y}^{\mathbb{Z}_3}  =\fratsm{x^3,xy,yx,y^3}.
$$
\end{example}

\section{\Unram groups and their invariants}\label{sec:unram}

The following is our main structure theorem for invariant fields of %totally unramified 
solvable groups.

	\begin{theorem}\label{thm:unram2}

Let $G\subset \GL_d(\FF)$ be a 
finite group acting on $\FF^d$ via a complete representation.
Then
$$\frats{x}{d}^G \cong \frats{u}{|G|(d-1)+1}.$$

	\end{theorem}

A key step in the proof of Theorem \ref{thm:unram2} will be the following lemma.

\begin{lemma}\label{lem:linear}
Let $\pi=\pi_B\oplus\pi_J$ be a complete representation of $G$ on $\FF^d$ and let $N$  be a normal abelian subgroup as in the definition of complete representation in Section \ref{ssec:complete_representations}. 
Then $G/N$ acts linearly on 
the free generators of $\frats{x}{d}^N$ constructed in Corollary \ref{cor: abelbasis} (when applied to the abelian group $N$).
\end{lemma}

\begin{proof}
Let $b_\chi$ and $j_k$  diagonalize $\pi_B|_N$ and $\pi_J|_N$, respectively.
Here $b_\chi$ are indexed by $N^*\backslash\{\tau\}$.
Then $ng\cdot b_\chi=\chi(g\inv n g)(g\cdot b_\chi)=\chi^g(n)(g\cdot b_\chi)$, so 
$g\cdot b_\chi$ is a scalar multiple of $b_{\chi^g}$.

Denote $V_\chi=\textrm{span }\{ j_i \colon n\cdot j_i=\chi(n)j_i\}$.
Since $ng\cdot j_i=\chi(g\inv n g)(g\cdot j_i)= \chi^g(n)(g\cdot j_i)$, $v\in V_\chi$ implies $g\cdot v \in V_{\chi^g}$.
\end{proof}

	\begin{proof}[Proof of Theorem \ref{thm:unram2}]
Let $\pi$ be a complete representation of $G$ and $N$ a nontrivial abelian normal subgroup corresponding to it. Then $G/N$ acts linearly on
$$\frats{x}{d}^N \cong \frats{u}{|N|(d-1)+1}$$
by Lemma \ref{lem:linear}; furthermore, by the description of the generators $u_j$ in Corollary \ref{cor: abelbasis}, this action is precisely the representation
\[\pi_{N^\tau}
\oplus (\pi_{B}\otimes \pi \oplus \pi \otimes \pi_B)_{N^\tau}
\oplus (\pi_{B}\otimes \pi \otimes \pi_B)_{N^\tau}.\]
Since it is a complete representation of $G/N$ by assumption,  induction implies
\[
\frats{x}{d}^G \cong \frats{u}{|N|(d-1)+1}^{G/N}\cong\frats{\tilde{u}}{|G|(d-1)+1}.\qedhere
\]
\end{proof}

\begin{proof}[Proof of Theorem \ref{thm:unram}]
We prove that the left regular representation of a totally \unram group is complete. 
Let $G$ be \unram over $N$. 
By Clifford's theorem  \cite[Theorem 6.2]{Isa76}, 
 irreducible representations of $N$ partition into orbits and each orbit is represented by an irreducible representation of $G$. 
Take the nontrivial representatives and define $\pi_B$ as their sum.
The left regular representation of $G/N$ is then contained in $\pi_{N^\tau}$ and we are done by recursion.
\end{proof}

\begin{example}
Define a representation of $S_3$ on $\mathbb{F}^3$ via $\sigma x_i=x_{\sigma(i)}$. 
The representation of the normal subgroup $N$ generated by $(1\ 2\ 3)$ is diagonalized in the basis
$v_1 =x_1+x_2+x_3,\ v_2=x_1+\omega x_2+ \omega^2 x_3,\ v_3=x_1+\omega^2 x_2+\omega x_3$, where $\omega$ is the third root of unity.
By Corollary \ref{cor: abelbasis}, 
we obtain the invariant skew field $\fratsm{x_1,x_2,x_3}^N=\fratsm{v_1,v_2v_3,v_3v_2,v_2v_1v_3,v_3v_1v_2,v_2^3,v_3^3}=\frats{z}{7}$. 

The action of $G/N\cong \mathbb Z_2$ on $\frats{z}{7}$  is given by the action of $(2\ 3)$ {(or any other transposition)} on  the initial variables. We get a representation given by
$$
z_1 \mapsto z_1,
z_2 \mapsto z_3,
z_3 \mapsto z_2,
z_4 \mapsto z_5,
z_5 \mapsto z_4,
z_6 \mapsto z_7,
z_7 \mapsto z_6,
$$
which is diagonalized by 
$$
w_1=z_1,
w_2=z_2+z_3,
w_3=z_2-z_3,
w_4=z_4+z_5,
w_5=z_4-z_5,
w_6=z_6+z_7,
w_7=z_6-z_7.
$$
Finally, applying Corollary \ref{cor: abelbasis} again, the obtained free generators of $\fratsm{x_1,x_2,x_3}^{S_3}$  are
$$
w_1,\ w_2,\  w_4,\  w_6,\ 
 w_3^2,\  w_3w_5,\  w_3w_7,\  w_5w_3,\  w_7w_3,\ 
w_3w_1w_3,\  w_3w_2w_3,\  w_3w_4w_3,\  w_3w_6w_3.
$$
\end{example}

\begin{example}\label{exa:unranoncom}

	Even though the standard two-dimensional representation of 
	$D_4= \mathbb{Z}_4 \rtimes \mathbb{Z}_2$ given by  $a\cdot x = i x,\ a\cdot y = -iy,\ b\cdot x =y$ and $b \cdot y = x$
	is not complete, we can still compute its invariants. 
	The invariants of $N=\langle a\rangle\cong\mathbb Z_4$
	are freely generated by
	$
	z_1=xy,\ z_2=yx,\ z_3=x^2y^2,\ z_4=y^2x^2
	$ 
	and 
	$ 
	z'_5=x^4. 
	$
	Then we replace $z'_5$ by $z_5=z'_5z_4\inv =x^2y^{-2}$.
	The action of $D_4/N\cong\mathbb{Z}_2$ on these generators is 
	$$
	z_1 \mapsto z_2, \ z_2 \mapsto z_1,\ z_3 \mapsto z_4,\ z_4 \mapsto z_3,\ z_5 \mapsto z_5\inv.
	$$
	Observe that this action is linearized and diagonalized with respect to
	$$
	w_1=z_1+z_2,\ w_2=z_1-z_2,\ w_3 =z_3+z_4,\ w_4=z_3-z_4,\ w_5=(1+z_5)(1-z_5)^{-1}.
	$$
		Finally we get nine free generators of the rational invariants of $D_4$:
	$$
	w_1,\  w_2^2,  \ w_2w_1w_2, \ w_2w_3w_2, \ w_2w_4w, \ w_2w_5,\ w_3,\ w_4w_2, \ w_5w_2.
	$$

\end{example}

\begin{example}
	The smallest example of a not totally \unram group is $\SL_2(\mathbb{F}_3)$. 
	It has only one nontrivial normal abelian subgroup $N\cong\mathbb{Z}_2$;  it is generated by $\diag(2,2)$. 
	Every irreducible representation restricted to $N$ is trivial or contains two copies of the sign representation.

	Let us describe problems arising in the computation of a generating set for the skew field of invariants. 
	We start with a two-dimensional irreducible representation of $\SL_2(\mathbb F_3)$ on $\mathbb F^2$, for instance defined by
	$$
	\begin{pmatrix}
	1 & 1 \\
	0 & 1
	\end{pmatrix}
	\mapsto
	\begin{pmatrix}
	0 & \omega^2 \\
	-\omega & -1
	\end{pmatrix},\quad
	\begin{pmatrix}
	0 & 1 \\
 	2	& 0
	\end{pmatrix}
	\mapsto
	\begin{pmatrix}
	0 & -\omega\\
	\omega^2 & 0
	\end{pmatrix}.
	$$
  	The generator of $N$ is mapped to $\diag(-1,-1)$.
	Hence the free generators of $N$-invariants are $x^2,\ xy$ and $yx$.

	The group $G/N$ has one abelian normal subgroup $\tilde N/N$; it is isomorphic to $\mathbb{Z}_2 \times \mathbb{Z}_2$. Representatives of its generators are mapped to 
	$$
	\begin{pmatrix}
	0 & \omega\\
	-\omega^2 & 0
	\end{pmatrix}
	\ \text{ and } \
	\begin{pmatrix}
	-\omega & -1\\
	\omega & \omega
	\end{pmatrix}.
	$$
	The action of these two on $N$-invariants is given by
	\begin{equation}\label{eq:act1}
	x^2 \mapsto \omega y^2=\omega(yx)(x^2)^{-1}(xy), \ xy \mapsto -yx, \ yx \mapsto -xy,
	\end{equation}
	and 
	\begin{equation}\label{eq:act2}
	x^2 \mapsto \omega (x+y)^2, \ xy \mapsto -\omega(x+y)(\omega y -x) , \ yx \mapsto -\omega(\omega y -x)(x+y).
	\end{equation}
	Now the problem is to find a set of free generators of $N$-invariants that simultaneously linearizes and diagonalizes both mappings as we have done in Example \ref{exa:unranoncom}. It is straightforward to linearize
	\eqref{eq:act1} by using a linear fractional transformation 
	in
	$xy^{-1}=x^2(yx)^{-1}$ (cf.~Example \ref{exa:unranoncom}), but then
	the action \eqref{eq:act2} becomes unwieldy. We have been unable to determine if $\mathbb{F}\plangle x,y\prangle^{\SL_2(\mathbb F_3)}$ is free (on 25 generators).
	\end{example}

\section{Positivity of invariant rational functions}\label{sec:pos}

In this section we investigate positive invariant noncommutative rational functions and prove an 
invariant rational Positivstellensatz in Theorem \ref{thm:extendrealization} for solvable groups $G$.
A finer structure of constraint positivity is proved in
Subsection \ref{ssec:qm}.
Positivity certificates for invariants
are ubiquitous in real algebraic geometry literature, see e.g.~\cite{PS85,CKS,Riener,Broecker}.
Throughout this section let $\mathbb{F}=\mathbb{C}$ be the field of complex numbers. We endow $\fratc{x}{d}$ with 
the natural involution fixing the $x_j$ and extending the complex conjugation on $\mathbb{C}$.

\begin{lemma}\label{lem:extendrealization}
Let $G \subset \ort$ be a finite solvable group, $H$ its normal subgroup, and assume that
$N= G/H$ is abelian. 
There exists 
an invertible matrix $R_N \in \M_{|N|}(\fratc{x}{d}^H)$ such that for every $Q_H \in \M_n(\fratc{x}{d}^H)$, 
\[Q_G =  (R_N\otimes I)^*\left(\bigoplus_{n\in N} n\cdot Q_H\right)(R_N\otimes I) \in \M_{|N|n}(\fratc{x}{d}^G)\]
and every $r \in \fratc{x}{d}^G$ of the form 
\begin{equation}\label{e:sos}
r(x) = c^*\left(\left(A_0 + \sum_i A_i q_i(x)\right)^{-1}\right)^*Q_H\left(A_0 + \sum_i A_i q_i(x)\right)^{-1}c,
\end{equation}
where the $q_i \in \fratc{x}{d}^H$, can be rewritten as 
\[r(x) = \tilde{c}^*\left(\left(\tilde{A}_0 + \sum_j \tilde{A}_j \tilde{q}_j(x)\right)^{-1}\right)^*
Q_G\left(\tilde{A}_0 + \sum_j \tilde{A}_j \tilde{q}_j(x)\right)^{-1}\tilde{c}\]
with $\tilde{q}_j\in \fratc{x}{d}^G.$
\end{lemma}

\begin{proof}
Consider $r(x)$ as in \eqref{e:sos}. Tracing through the proof of Theorem \ref{t:solv} we see that $r(x)$ admits a realization
$$
\resizebox{1.0 \textwidth}{!} {
$
\tilde{c}^*\left(\left(\tilde{A}_0 + \sum_j \tilde{A}_j \tilde{q}_j(x)\right)^{-1}\right)^*
((\hat{M} \Gamma^*)\otimes I) 
\left(\bigoplus_{n\in N} n\cdot Q_H\right)
((\Gamma M)\otimes I)\left(\tilde{A}_0 + \sum_j \tilde{A}_j \tilde{q}_j(x)\right)^{-1}\tilde{c},
$
}
$$
where $\Gamma$ is a unitary change of basis matrix (more precisely, columns of $\Gamma^*$ are eigenvectors for the left regular representation of $N^*$). Note that we can take $\hat{M} = M^*$. Hence $R_N=\Gamma M$ is the desired matrix.
\end{proof}

\begin{theorem}\label{thm:extendrealization}
Let $G \subset\ort$ be a finite solvable group. There exists 
an invertible matrix $R_G  \in \M_{|G|}(\fratc{x}{d})$ such that for every $Q \in \M_n(\fratc{x}{d})$,
\[Q_G =  (R_G\otimes I)^*\left(\bigoplus_{g\in G} g\cdot Q\right)(R_G\otimes I)\in \M_{|G|n}(\fratc{x}{d}^G)\]
and  every $r \in \fratc{x}{d}^G$ of the form 
\begin{equation}
\label{e:62}
r(x) = c^*\left(\left(A_0 + \sum_i A_i q_i(x)\right)^{-1}\right)^*Q\left(A_0 + \sum_i A_i q_i(x)\right)^{-1}c,
\end{equation}
where the $q_i \in \fratc{x}{d},$
can be rewritten as 
\begin{equation}
\label{e:63}
r(x) = \tilde{c}^*\left(\left(\tilde{A}_0 + \sum_j \tilde{A}_j \tilde{q}_j(x)\right)^{-1}\right)^*
Q_G\left(\tilde{A}_0 + \sum_j \tilde{A}_j \tilde{q}_j(x)\right)^{-1}\tilde{c}
\end{equation}
with $\tilde{q}_j\in \fratc{x}{d}^G.$
\end{theorem}

\begin{proof}
Apply Lemma \ref{lem:extendrealization} and induction on the derived series  of $G$.
\end{proof}

\begin{corollary}\label{cor:sos1st}
Let $G \subset\ort$ be a finite solvable group.
Then there are $q_1, \ldots, q_N \in \fratc{x}{d}^G$
such that for every $r \in \fratc{x}{d}^G,$
if
\(r = \sum_i s_i^*s_i,\)
then
\[r = \sum_j \tilde{s}_j^* q_{n_j} \tilde{s}_j,\]
where $\tilde{s}_j \in \fratc{x}{d}^G.$
\end{corollary}

\begin{proof}
If $s_i = c_i^* L_i^{-1}b_i$ is a realization of $s_i$, then
\begin{align*}
s_i^*s_i &=
\begin{pmatrix}0 & b_i^*\end{pmatrix}
\begin{pmatrix}c_ic_i^* & L_i^* \\ -L_i & 0\end{pmatrix}^{-1}
\begin{pmatrix}0 \\ b_i\end{pmatrix} \\
& = \frac12\begin{pmatrix}0 & b_i^*\end{pmatrix}
\left(\begin{pmatrix}c_ic_i^* & L_i^* \\ -L_i & 0\end{pmatrix}^{-1}
+\begin{pmatrix}c_ic_i^* & -L_i^* \\ L_i & 0\end{pmatrix}^{-1}\right)
\begin{pmatrix}0 \\ b_i\end{pmatrix} \\
& =\frac12\begin{pmatrix}0 & b_i^*\end{pmatrix}
\begin{pmatrix}c_ic_i^* & L_i^* \\ -L_i & 0\end{pmatrix}^{-1}
\left(\begin{pmatrix}c_ic_i^* & L_i^* \\ -L_i & 0\end{pmatrix}
+\begin{pmatrix}c_ic_i^* & -L_i^* \\ L_i & 0\end{pmatrix}\right)
\begin{pmatrix}c_ic_i^* & -L_i^* \\ L_i & 0\end{pmatrix}^{-1}
\begin{pmatrix}0 \\ b_i\end{pmatrix} \\
& =\begin{pmatrix}0 & b_i^*\end{pmatrix}
\begin{pmatrix}c_ic_i^* & L_i^* \\ -L_i & 0\end{pmatrix}^{-1}
\begin{pmatrix}
c_ic_i^* &0 \\ 0 & 0
\end{pmatrix}
\begin{pmatrix}c_ic_i^* & -L_i^* \\ L_i & 0\end{pmatrix}^{-1}
\begin{pmatrix}0 \\ b_i\end{pmatrix}.
\end{align*}
Therefore $r$ can be written as in \eqref{e:62} for a constant positive semidefinite $Q=P^*P$. By Theorem \ref{thm:extendrealization}, $r$ can be written as in \eqref{e:63} with $Q_G=(R_G\otimes P)^*(R_G\otimes P)$, which then yields the desired $G$-invariant sum of hermitian squares presentation for $r$.
\end{proof}

Recall that a rational function $r$ is {\bf positive} if for every $n\in\N$ and $X=X^*\in\M_n(\C)^d$, $r$ is defined at $X$ and $r(X)$ is positive semidefinite. We are now ready to prove Theorem \ref{c:h17}.

\begin{proof}[Proof of Theorem \ref{c:h17}]
Since $r$ is positive semidefinite, it is a sum of hermitian squares by \cite[Theorem 4.5]{KPV}. The conclusion now follows
from Corollary \ref{cor:sos1st}.
\end{proof}

\def\rex{\mathfrak r} 

\subsection{Quadratic modules and free semialgebraic sets}\label{ssec:qm}
The {\bf quadratic module}  
associated to a symmetric matrix $Q\in \M_n(\fratc{x}{d})$
is
\[
\resizebox{1.0 \textwidth}{!} 
{
    $
\QM_{\fratc{x}{d}}(Q)=\Big\{\sum_i u_i^*u_i+ \sum_{j}v_j^* Q v_j \colon v_j\in{\fratc{x}{d}}^n, \, u_i\in \fratc{x}{d}\Big\}.
$
}
\]
It is called {\bf Archimedean} if there is $N\in\N$ so that $N-\sum_j x_j^2\in \QM_{\fratc{x}{d}}(Q)$.
The associated {\bf free semialgebraic set} is
\[
\cD_Q=
\big\{ X=X^*\in B(\cH)^d\colon Q(X)\succeq0\big\},
\]
where $\cH$ is a separable Hilbert space.

\begin{corollary}\label{cor:quadraticmodule}
Let $G \subset\ort$ be a finite solvable group. 
Then
\[\QM_{\fratc{x}{d}}(Q)\cap \fratc{x}{d}^G = \QM_{\fratc{x}{d}^G}(Q_G).\]
\end{corollary}

\begin{proof}
	Use Theorem \ref{thm:extendrealization}.
\end{proof}

Let $\rex$ be a formal rational expression. We say that $\rex$ is {\bf (strictly) positive} on $\cD_Q$ if for every $X\in\cD_Q$, $\rex$ is defined at $X$ and $\rex(X)$ is a positive semidefinite (definite) operator. In this case we write $\rex\succeq0$ ($\rex\succ0$) on $\cD_Q$.

The reason for using formal rational expressions is that rational functions (as elements of the free skew field) do not admit unambiguous evaluations on $B(\cH)^d$. For example, the expression $\rex=x_1(x_2x_1)^{-1}x_2-1$ represents the zero element of the free skew field, but admits nonzero evaluations on operators, namely $\rex(S,S^*)\neq0$ where $S$ is the unilateral shift on $\ell^2(\N)$.

\begin{corollary}\label{prop:posSS}
Let $G \subset\ort$ be a finite solvable group. 
Suppose $Q=Q^*\in\M_n(\fpolyc{x}{d})$ is such that $\QM_{\fratc{x}{d}}(Q)$ is Archimedean.
If $\rex$ is a formal rational expression such that $\rex \succ 0$ on $\cD_Q$ and $\rex$ induces a $G$-invariant rational function $r$, then
$r\in \QM_{\fratc{x}{d}^G}(Q_G).$
\end{corollary}

\begin{proof}
By \cite[Theorem 2.1]{PascoePAMS}, $r\in \QM_{\fratc{x}{d}}(Q)$. Now Corollary \ref{cor:quadraticmodule} finishes the proof.
\end{proof}

When the semialgebraic set $\cD_Q$ is convex, in which case one can assume that $Q$ is a symmetric affine matrix by the renowned Helton--McCullough theorem \cite{HM12}, one can certify (non-strict) positivity on $\cD_Q$. 

\begin{corollary}\label{prop:ConvexPosSS}
Let $G \subset\ort$ be a finite solvable group. 
Assume $Q=Q^* \in\M_n(\fpolyc{x}{d})$ is linear with $Q(0)=I$.
If $\rex$ is a formal rational expression such that $\rex \succeq 0$ on $\cD_Q$ and $\rex$ induces a $G$-invariant rational function $r$, then
$r\in \QM_{\fratc{x}{d}^G}(Q_G).$
\end{corollary}

\begin{proof}
By \cite[Theorem 3.1]{PascoePAMS}, $r\in \QM_{\fratc{x}{d}}(Q)$. Now apply Corollary \ref{cor:quadraticmodule}.
\end{proof}

\begin{example}
Let $G=S_2$ act on $\fratcm{x,y}$. Then
$$a=x+y,\quad b=(x-y)^2,\quad c=(x-y)(x+y)(x-y)$$
are free generators of $\fratcm{x,y}^G$. The matrix $R_G$ from Theorem \ref{thm:extendrealization} equals
$$R_G
=\frac{1}{\sqrt{2}}
\begin{pmatrix}
1 & 1 \\ -1 & 1
\end{pmatrix}\diag(1,x-y)
=\frac{1}{\sqrt{2}}
\begin{pmatrix}
1 & x-y \\ -1 & x-y
\end{pmatrix}.$$
By computing the invariant middle matrix $Q_G$ we obtain the following Positivstellens\"atze.
\begin{enumerate}
\item (Entire space) If $Q=1$ then $Q_G=\diag(1,b)$. By Theorem \ref{c:h17}, every positive $G$-invariant rational function $r$ is of the form
$$r=\sum_ju_j^*u_j+\sum_jv_jbv_j^*,\qquad u_j,v_j\in\fratcm{x,y}^G.$$

\item (Disk) If $Q=1-x^2-y^2$ then
$$Q_G=\diag\left(1-\tfrac12 (a^2+b),b-\tfrac12 (cb^{-1}c+b^2)\right).$$
Since $Q$ clearly generates an Archimedean quadratic module, every $G$-invariant rational expression strictly positive on the disk
$$\{(X,Y)\colon X^2+Y^2\preceq I \}$$
induces a rational function of the form
$$\sum_ju_j^*u_j+\sum_jv_j(1-\tfrac12 (a^2+b))v_j^*+\sum_jw_j(b-\tfrac12 (cb^{-1}c+b^2))w_j^*$$
for $u_j,v_j,w_j\in\fratcm{x,y}^G$ by Corollary \ref{prop:posSS}. On the other hand, the disk also admits a linear matrix representation given by
$$Q'=\begin{pmatrix}
1 & 0 & x \\ 0& 1 & y \\ x& y & 1
\end{pmatrix},$$
which yields
$$Q'_G=\frac12
\begin{pmatrix}
2 & 0 & a & 0 & 0 & b \\
0 & 2 & a & 0 & 0 & -b \\
a & a & 2 & b & -b & 0\\
0 & 0 & b & 2b & 0 & c \\
0 & 0 & -b & 0 & 2b & c \\
b & -b & 0 & c & c & 2b
\end{pmatrix}.$$
Note that the free semialgebraic set associated with $Q'_G$ as a matrix in variables $a,b,c$ is also convex. By Corollary \ref{prop:ConvexPosSS} we can use $Q'_G$ to describe $G$-invariant positivity on the disk.

\item (Bidisk) If $Q=\diag(1-x^2,1-y^2)$ then
$$Q_G=\frac12
\begin{pmatrix}
2-\tfrac12(a^2+b) & 0 &-\tfrac12(c+ab) & 0\\
0 & 2-\tfrac12(a^2+b) & 0 &\tfrac12(c+ab)\\
-\tfrac12(c+ba) &0 & 2b-\tfrac12 (cb^{-1}c+ b^2) & 0 \\
0& \tfrac12(c+ba) & 0& 2b-\tfrac12 (cb^{-1}c+b^2)
\end{pmatrix}.$$
Note that $2Q_G$ is unitarily similar to a direct sum of two copies of
$$S=\begin{pmatrix}
2-\tfrac12(a^2+b) & \tfrac12(c+ab) \\ \tfrac12(c+ba) & 2b-\tfrac12 (cb^{-1}c+b^2)
\end{pmatrix}.$$
Every $G$-invariant rational expression strictly positive on the bidisk
$$\{(X,Y)\colon X^2\preceq I \ \&\  Y^2\preceq I \}$$
induces a rational function in $\QM_{\fratcm{x,y}^G}(S)$ by Corollary \ref{prop:posSS}. As in the case of the disk, bidisk can also be represented by a monic linear matrix inequality, which by Corollary \ref{prop:ConvexPosSS} then gives a description of invariant expressions positive on the bidisk.

\item (Positive orthant) If $Q=\diag(x,y)$ then $2Q_G$ is unitarily similar to a direct sum of two copies of
$$S=\begin{pmatrix}
a & b \\
b & c
\end{pmatrix}.$$
The positive orthant
$$\{(X,Y)\colon X\succeq0 \ \& \ Y\succeq0 \}$$
is a convex semialgebraic set, and after a scalar shift it admits a monic linear matrix representation. Hence rational expressions positive on the orthant induce rational functions in $\QM_{\fratcm{x,y}}(Q)$ by \cite[Theorem 3.1]{PascoePAMS}. The $G$-invariant rational functions among them then lie to $\QM_{\fratcm{x,y}^G}(S)$ by Corollary \ref{cor:quadraticmodule}.

\end{enumerate}
\end{example}

\def\om{\omega}
\begin{example}
Let $G=\mathbb{Z}_3$ act on $\fratcm{x,y,z}$. Let $\om=-\frac12+i\frac{\sqrt{3}}{2}$ and
$$q_1=\om x+\om^2 y+z,\quad q_2=\om^2 x+\om y+z.$$
Then
$$R_G=\frac{1}{\sqrt{3}}
\begin{pmatrix}
1 & 1 & 1 \\
\om & \om^2 & 1 \\
\om^2 & \om & 1 \\
\end{pmatrix}\diag(1,q_2,q_1).$$
For $Q=1$ we get $Q_G=\diag(1,q_1q_2,q_2q_1)$. Therefore all positive semidefinite $G$-invariant rational functions in $\fratcm{x,y,z}$ are of the form
$$\sum_ju_j^*u_j+\sum_jv_jq_1q_2v_j^*+\sum_jw_jq_2q_1w_j^*,
\qquad u_j,v_j,w_j\in\fratcm{x,y,z}^G.$$
\end{example}

\bibliography{references}

\newcommand{\etalchar}[1]{$^{#1}$}
\begin{thebibliography}{CHKK10}

\bibitem[AM72]{AM72}
Michael Artin and David~B. Mumford.
\newblock Some elementary examples of unirational varieties which are not
  rational.
\newblock {\em Proc. London Math. Soc. (3)}, 25:75--95, 1972.

\bibitem[AMY18]{AMY18}
Jim Agler, John~E. McCarthy, and N.~J. Young.
\newblock Non-commutative manifolds, the free square root and symmetric
  functions in two non-commuting variables.
\newblock {\em Trans. London Math. Soc.}, 5(1):132--183, 2018.

\bibitem[AY14]{AY14}
Jim Agler and Nicholas~J. Young.
\newblock Symmetric functions of two noncommuting variables.
\newblock {\em J. Funct. Anal.}, 266(9):5709--5732, 2014.

\bibitem[AY17]{AY17}
Jim Agler and Nicholas~J. Young.
\newblock Realization of functions on the symmetrized bidisc.
\newblock {\em J. Math. Anal. Appl.}, 453(1):227--240, 2017.

\bibitem[BGM05]{BGM05}
Joseph~A. Ball, Gilbert Groenewald, and Tanit Malakorn.
\newblock Structured noncommutative multidimensional linear systems.
\newblock {\em SIAM J. Control Optim.}, 44(4):1474--1528, 2005.

\bibitem[BGM06]{BGM06}
Joseph~A. Ball, Gilbert Groenewald, and Tanit Malakorn.
\newblock Conservative structured noncommutative multidimensional linear
  systems.
\newblock In {\em The state space method generalizations and applications},
  volume 161 of {\em Oper. Theory Adv. Appl.}, pages 179--223. Birkh\"{a}user,
  Basel, 2006.

\bibitem[BR11]{BR11}
Jean Berstel and Christophe Reutenauer.
\newblock {\em Noncommutative rational series with applications}, volume 137 of
  {\em Encyclopedia of Mathematics and its Applications}.
\newblock Cambridge University Press, Cambridge, 2011.

\bibitem[Bro98]{Broecker}
Ludwig Broecker.
\newblock On symmetric semialgebraic sets and orbit spaces.
\newblock In {\em Singularities {S}ymposium---\L ojasiewicz 70 ({K}rak\'{o}w,
  1996; {W}arsaw, 1996)}, volume~44 of {\em Banach Center Publ.}, pages 37--50.
  Polish Acad. Sci. Inst. Math., Warsaw, 1998.

\bibitem[BRRZ08]{BRRZ08}
Nantel Bergeron, Christophe Reutenauer, Mercedes Rosas, and Mike Zabrocki.
\newblock Invariants and coinvariants of the symmetric groups in noncommuting
  variables.
\newblock {\em Canad. J. Math.}, 60(2):266--296, 2008.

\bibitem[CHHK15]{CHHK15}
Huah Chu, Akinari Hoshi, Shou-Jen Hu, and Ming-chang Kang.
\newblock Noether's problem for groups of order 243.
\newblock {\em J. Algebra}, 442:233--259, 2015.

\bibitem[CHKK10]{CHKK10}
Huah Chu, Shou-Jen Hu, Ming-chang Kang, and Boris~E. Kunyavskii.
\newblock Noether's problem and the unramified {B}rauer group for groups of
  order 64.
\newblock {\em Int. Math. Res. Not. IMRN}, (12):2329--2366, 2010.

\bibitem[CKS09]{CKS}
Jaka Cimpri\v{c}, Salma Kuhlmann, and Claus Scheiderer.
\newblock Sums of squares and moment problems in equivariant situations.
\newblock {\em Trans. Amer. Math. Soc.}, 361(2):735--765, 2009.

\bibitem[Coh95]{Coh95}
Paul~Moritz Cohn.
\newblock {\em Skew Fields: Theory of General Division Rings}.
\newblock Encyclopedia of Mathematics and its Applications. Cambridge
  University Press, Cambridge, 1995.

\bibitem[Coh06]{Coh06}
Paul~Moritz Cohn.
\newblock {\em Free ideal rings and localization in general rings}, volume~3 of
  {\em New Mathematical Monographs}.
\newblock Cambridge University Press, Cambridge, 2006.

\bibitem[CPTD18]{CPT-D18}
David Cushing, James~E. Pascoe, and Ryan Tully-Doyle.
\newblock Free functions with symmetry.
\newblock {\em Math. Z.}, 289(3-4):837--857, 2018.

\bibitem[CTS07]{C-TS07}
Jean-Louis Colliot-Th\'el\`ene and Jean-Jacques Sansuc.
\newblock The rationality problem for fields of invariants under linear
  algebraic groups (with special regards to the {B}rauer group).
\newblock In {\em Algebraic groups and homogeneous spaces}, volume~19 of {\em
  Tata Inst. Fund. Res. Stud. Math.}, pages 113--186. Tata Inst. Fund. Res.,
  Mumbai, 2007.

\bibitem[DK15]{DK15}
Harm Derksen and Gregor Kemper.
\newblock {\em Computational invariant theory}, volume 130 of {\em
  Encyclopaedia of Mathematical Sciences}.
\newblock Springer, Heidelberg, enlarged edition, 2015.
\newblock With two appendices by Vladimir L. Popov, and an addendum by Norbert
  A'Campo and Popov, Invariant Theory and Algebraic Transformation Groups,
  VIII.

\bibitem[GKL{\etalchar{+}}95]{GKLLRT95}
Israel~M. Gelfand, Daniel Krob, Alain Lascoux, Bernard Leclerc, Vladimir~S.
  Retakh, and Jean-Yves Thibon.
\newblock Noncommutative symmetric functions.
\newblock {\em Adv. Math.}, 112(2):218--348, 1995.

\bibitem[GM85]{GM85}
Jessy~W. Grizzle and Steven~I. Marcus.
\newblock The structure of nonlinear control systems possessing symmetries.
\newblock {\em IEEE Trans. Automat. Control}, 30(3):248--258, 1985.

\bibitem[GR08]{GR08}
Bruce Gilligan and Guy~J. Roos, editors.
\newblock {\em Symmetries in complex analysis}, volume 468 of {\em Contemporary
  Mathematics}.
\newblock American Mathematical Society, Providence, RI, 2008.
\newblock Lectures from the Workshop on Several Complex Variables, Analysis on
  Complex Lie Groups and Homogeneous Spaces held at Zhejiang University,
  Hangzhou, October 17--29, 2005.

\bibitem[HM12]{HM12}
J.~William Helton and Scott McCullough.
\newblock Every convex free basic semi-algebraic set has an {LMI}
  representation.
\newblock {\em Ann. of Math. (2)}, 176(2):979--1013, 2012.

\bibitem[HMV06]{HMV06}
J.~William Helton, Scott~A. McCullough, and Victor Vinnikov.
\newblock Noncommutative convexity arises from linear matrix inequalities.
\newblock {\em J. Funct. Anal.}, 240(1):105--191, 2006.

\bibitem[Isa76]{Isa76}
I.~Martin Isaacs.
\newblock {\em Character theory of finite groups}.
\newblock Academic Press [Harcourt Brace Jovanovich, Publishers], New
  York-London, 1976.
\newblock Pure and Applied Mathematics, No. 69.

\bibitem[JS]{JS+}
Urban Jezernik and Jonatan Sánchez.
\newblock Irrationality of generic quotient varieties via {B}ogomolov
  multipliers.
\newblock {\em preprint}, arXiv:1811.01851.

\bibitem[KPV17]{KPV}
Igor Klep, James~Eldred Pascoe, and Jurij Vol\v{c}i\v{c}.
\newblock Regular and positive noncommutative rational functions.
\newblock {\em J. Lond. Math. Soc. (2)}, 95(2):613--632, 2017.

\bibitem[Kra84]{Kra84}
Hanspeter Kraft.
\newblock {\em Geometrische {M}ethoden in der {I}nvariantentheorie}.
\newblock Aspects of Mathematics, D1. Friedr. Vieweg \& Sohn, Braunschweig,
  1984.

\bibitem[KVV12]{KVV12}
Dmitry~S. Kaliuzhnyi-Verbovetskyi and Victor Vinnikov.
\newblock Noncommutative rational functions, their difference-differential
  calculus and realizations.
\newblock {\em Multidimens. Syst. Signal Process.}, 23(1-2):49--77, 2012.

\bibitem[Kwa95]{Kwa95}
Huibert Kwakernaak.
\newblock Symmetries in control system design.
\newblock In {\em Trends in control ({R}ome, 1995)}, pages 17--51. Springer,
  Berlin, 1995.

\bibitem[Lew74]{Lew74}
Jacques Lewin.
\newblock Fields of fractions for group algebras of free groups.
\newblock {\em Trans. Amer. Math. Soc.}, 192:339--346, 1974.

\bibitem[Lin00]{Lin00}
Peter~A. Linnell.
\newblock A rationality criterion for unbounded operators.
\newblock {\em J. Funct. Anal.}, 171(1):115--123, 2000.

\bibitem[LS01]{LS01}
Roger~C. Lyndon and Paul~E. Schupp.
\newblock {\em Combinatorial group theory}.
\newblock Classics in Mathematics. Springer-Verlag, Berlin, 2001.
\newblock Reprint of the 1977 edition.

\bibitem[Mor12]{Mor12}
Primo\v{z} Moravec.
\newblock Unramified {B}rauer groups of finite and infinite groups.
\newblock {\em Amer. J. Math.}, 134(6):1679--1704, 2012.

\bibitem[Pas18]{PascoePAMS}
James~E. Pascoe.
\newblock Positivstellens\"{a}tze for noncommutative rational expressions.
\newblock {\em Proc. Amer. Math. Soc.}, 146(3):933--937, 2018.

\bibitem[Pey08]{Pey08}
Emmanuel Peyre.
\newblock Unramified cohomology of degree 3 and {N}oether's problem.
\newblock {\em Invent. Math.}, 171(1):191--225, 2008.

\bibitem[PS85]{PS85}
Claudio Procesi and Gerald Schwarz.
\newblock Inequalities defining orbit spaces.
\newblock {\em Invent. Math.}, 81(3):539--554, 1985.

\bibitem[Rie16]{Riener}
Cordian Riener.
\newblock Symmetric semi-algebraic sets and non-negativity of symmetric
  polynomials.
\newblock {\em J. Pure Appl. Algebra}, 220(8):2809--2815, 2016.

\bibitem[RS06]{RS06}
Mercedes~H. Rosas and Bruce~E. Sagan.
\newblock Symmetric functions in noncommuting variables.
\newblock {\em Trans. Amer. Math. Soc.}, 358(1):215--232, 2006.

\bibitem[Rud90]{Rud90}
Walter Rudin.
\newblock {\em Fourier analysis on groups}.
\newblock Wiley Classics Library. John Wiley \& Sons, Inc., New York, 1990.
\newblock Reprint of the 1962 original, A Wiley-Interscience Publication.

\bibitem[Sal84]{Sal84}
David~J. Saltman.
\newblock Noether's problem over an algebraically closed field.
\newblock {\em Invent. Math.}, 77(1):71--84, 1984.

\bibitem[Sat14]{satake2014algebraic}
Ichir{\^o} Satake.
\newblock {\em Algebraic structures of symmetric domains}.
\newblock Princeton University Press, 2014.

\bibitem[Ser77]{Ser77}
Jean-Pierre Serre.
\newblock {\em Linear representations of finite groups}.
\newblock Springer-Verlag, New York-Heidelberg, 1977.
\newblock Translated from the second French edition by Leonard L. Scott,
  Graduate Texts in Mathematics, Vol. 42.

\bibitem[Stu08]{Stu08}
Bernd Sturmfels.
\newblock {\em Algorithms in invariant theory}.
\newblock Texts and Monographs in Symbolic Computation. Springer-Verlag,
  Vienna, second edition, 2008.

\bibitem[vdS87]{vdS87}
A.~J. van~der Schaft.
\newblock Symmetries in optimal control.
\newblock {\em SIAM J. Control Optim.}, 25(2):245--259, 1987.

\bibitem[Vol18]{Vol18}
Jurij Vol\v{c}i\v{c}.
\newblock Matrix coefficient realization theory of noncommutative rational
  functions.
\newblock {\em J. Algebra}, 499:397--437, 2018.

\end{thebibliography}

\bibliographystyle{alpha}

\end{document}